\documentclass[conference]{IEEEtran}
\IEEEoverridecommandlockouts
\usepackage{needspace}
\usepackage{theorem}
\usepackage[ruled,linesnumbered]{algorithm2e}
\usepackage{cite}
\usepackage{color}
\usepackage{mysymbol}
\makeatletter
\def\ps@headings{%
\def\@oddhead{\mbox{}\scriptsize\rightmark \hfil \thepage}%
\def\@evenhead{\scriptsize\thepage \hfil \leftmark\mbox{}}%
\def\@oddfoot{}%
\def\@evenfoot{}}
\makeatother
\usepackage{amsfonts, amsmath}
\usepackage{url}
\usepackage{amsfonts}
\usepackage{amssymb}
\usepackage{acronym}
\usepackage{color}
\usepackage{graphicx}

\newcommand{\field}[1]{\mathbb{#1}}

\def\l{\lambda}

\def\A{\mathcal{A}}

\def\1{\mathbf{1}}

\newtheorem{corollary}{Corollary}

\newtheorem{proposition}{Proposition}

\begin{document}

%
%

\title{\vspace{.1in}
\huge Accelerated Backpressure Algorithm  \thanks{This research is supported by Army Research Lab MAST
Collaborative Technology Alliance, AFOSR complex networks program, ARO P-57920-NS, NSF CAREER CCF-0952867, and NSF CCF-1017454, ONR MURI N000140810747 and NSF-ECS-0347285.}}

\author{Michael Zargham$^{\dagger}$, Alejandro Ribeiro$^{\dagger}$, Ali Jadbabaie$^{\dagger}$ \thanks{$^{\dagger}$Michael Zargham, Alejandro Ribeiro and Ali Jadbabaie are with the Department of Electrical and Systems Engineering, University of Pennsylvania.}}

\maketitle

\begin{abstract}
We develop an Accelerated Back Pressure (ABP) algorithm using Accelerated Dual Descent (ADD), a distributed approximate Newton-like algorithm that only uses local information. Our construction is based on writing the backpressure algorithm as the solution to a network feasibility problem solved via stochastic dual subgradient descent.  We apply stochastic ADD in place of the stochastic gradient descent algorithm.  We prove that the ABP algorithm guarantees stable queues.  Our numerical experiments demonstrate a significant improvement in convergence rate, especially when the packet arrival statistics vary over time.
\end{abstract}

%
\section{Introduction}
This paper considers the problem of joint routing and scheduling in packet networks. Packets are accepted from upper layers as they are generated and marked for delivery to intended destinations. To accomplish delivery of information nodes need to determine routes and schedules capable of accommodating the generated traffic. From a node-centric point of view, individual nodes handle packets that are generated locally as well as packets received from neighboring nodes. The goal of each node is to determine suitable next hops for each flow conducive to successful packet delivery.

A joint solution to this routing and scheduling problem is offered by the backpressure (BP) algorithm \cite{backpress}. In BP nodes keep track of the number of packets in their local queues for each flow and share this information with neighboring agents. The differences in the number of queued packets at two neighboring terminals are computed for all flows and the transmission capacity of the link is assigned to the flow with the largest queue differential. We can interpret this algorithm by identifying queue differentials with pressure to send packets on a link. Regardless of interpretation and despite its simplicity, BP can be proved to be an optimal policy if given arrival rates can be supported by some routing-scheduling policy, they can be supported by BP. Notice that BP relies on queue lengths only and does not need knowledge or estimation of packet arrival rates. It is also important to recognize that BP is an iterative algorithm. Packets are initially routed more or less at random but as queues build throughout the network suitable routes and schedules are eventually learned. 

A drawback of BP is the slow convergence rate of this iterative process for route discovery. This is better understood through an alternative interpretation of BP as a dual stochastic subgradient descent algorithm \cite{crosslayer, neelymod}. Joint routing and scheduling can be formulated as the determination of per-flow routing variables that satisfy link capacity and flow conservation constraints. In this model the packet transmission rates are abstracted as continuous variables. To solve the resulting feasibility problem we associate Lagrange multipliers with the flow conservation constraints and proceed to implement subgradient descent in the dual domain. This solution methodology leads to distributed implementations and for that reason is commonly utilized to find optimal operating points of wired \cite{kelly, low, Srikant} and wireless communication networks \cite{mung, energy, sep}. More important to the discussion here, the resulting algorithm turns out equivalent to BP with queue lengths taking the place of noisy versions of the corresponding dual variables. The slow convergence rate of BP is thus expected because the convergence rate of subgradient descent algorithms is logarithmic in the number of iterations \cite{ratejournal}. Simple modifications can speed up the convergence rate of BP by rendering it equivalent to stochastic gradient descent \cite{sbp} -- as opposed to stochastic subgradient descent. Nevertheless, the resulting linear convergence rate is still a limiting factor in practice.

To speed up the convergence rate of BP we need to incorporate information on the curvature of the dual function. This could be achieved by using Newton's method, but the determination of dual Newton steps requires coordination among all nodes in the network. To overcome this limitation, methods to determine approximate Newton steps in a distributed manner have been proposed. Early contributions on this regard are found in \cite{BeG83, cgNewt}. Both of these. Both of these methods, however, are not fully distributed because they require some level of global coordination. Efforts to overcome this shortcoming include approximating the Hessian inverse with the inverse of its diagonals \cite{lowdiag}, the use of consensus iterations to approximate the Newton step \cite{cdc09, wei},  and the accelerated dual descent (ADD) family of algorithms that uses a Taylor's expansion of the Hessian inverse to compute approximations to the Newton step \cite{mz12}. Members of the ADD family are indexed by a parameter $N$ indicating that local Newton step approximations are constructed at each node with information gleaned from agents no more than $N$ hops away. Algorithms ADD-$N$ were proven to achieve quadratic convergence rate and demonstrated to reduce the time to find optimal operating points -- as measured by the number of communication instances -- by two orders of magnitude with respect to regular subgradient descent. It has been shown that stochastic ADD converges for the Network flow optimization in \cite{ADDstoch}.The goal of this paper is to adapt the ADD family of algorithm to develop variations of BP that achieve quadratic convergence.

We start the paper introducing the problem of stabilizing queues in a communication network and review the backpressure algorithm, \cite{backpress} used to solve this problem (Section \ref{prelim}).  We proceed to demonstrate that the backpressure algorithm is equivalent to stochastic subgradient descent (Section \ref{ssgd}).  Once in the optimization framework we review the generalization to soft backpressure (Section \ref{ssbp}).  In Section \ref{abps}, we construct the accelerated backpressure algorithm, which is an approximation of Newton's method using local information and limited communication with neighboring nodes. In Section \ref{analysis}, we prove that the accelerated backpressure algorithm stabilizes the queues. In Section \ref{sim}, we present numerical experiments which demonstrate the performance gains associated with implementing the accelerated backpressure algorithm as compared to the backpressure algorithm from \cite{back press} and soft backpressure algorithm from \cite{sbp}.

The main contribution of this work is the introduction of a locally computable second order method for solving the queue stabilization problem.  This is particular relevant for cases where the packet arrival statistics vary with time. As shown in Figure \ref{queues}, the accelerated backpressure algorithm can effectively stabilize queues in networks whose arrival rates vary at higher frequencies than its first order parent algorithms.
%
\begin{algorithm}[t]\label{BPalg}{\small
\caption{Backpressure at node $i$}
\For{$t = 0, 1, 2, \cdots$} {
Observe local queue lengths $\{q^k_i(t)\}_{k}$ for all flows $k$\\
\For{all neighbors $j\in n_i$}{ 
Send queue lengths $\{q^k_i(t)\}_{k}$ -- Receive lengths $\{q^k_j(t)\}_{k}$\\
Determine flow with largest pressure:
$$k_{ij}^*=\argmax_k \left[q_i^k(t)-q_j^k(t)\right]^+ $$\\
Set routing variables to $r_{ij}^{k}(t) = 0$ for all $k\neq k_{ij}^*$ and
  $$r_{ij}^{k_{ij}^*} (t)= C_{ij} \ind{q_i^{k_{ij}^*}(t)-q_j^{k_{ij}^*}(t)>0} $$
Transmit $r_{ij}^{k_{ij}^*}(t)$ packets for flow $k_{ij}^*$}
}}
\end{algorithm}
%

\section{ Preliminaries}\label{prelim}
Consider a given a network $\mathcal{G}= \{\mathcal{V}, \mathcal{E}\}$ where $\mathcal{V}$ is the set of nodes and $\mathcal{E}\subseteq \mathcal{V}\times\mathcal{V}$ is the set of links between nodes. Denote as $C_{ij}$ the capacity of link $(i,j)\in \mathcal{E}$ and define the neighborhood of $i$ as the set $n_i=\{j\in \mathcal{V} | (i,j)\in \mathcal{E}\}$ of nodes $j$ that can communicate directly with $i$. There is also a set of information flows $\ccalK $ with the destination of flow $k\in\ccalK $ being the node $o_k\in \mathcal{V}$. 

At time index $t$ terminal $i\neq o_k$ generates a random number $a_i^k(t)$ of units of information to be delivered to $o_k$. The random variables $a_i^k(t)\geq0$ are assumed independent and identically distributed across time with expected value $\E{a_i^k(t)}=a_i^k$. In the same time slot node $i$ routes $r_{ij}^k(t)\geq0$ units of information through neighboring node $j\in n_i$ and receives $r_{ji}^k(t)\geq0$ packets from neighbor $j$. The difference between the total number of received packets $a_i^k(t) + \sum_{j\in n_i} r_{ji}^k(t) $ and the sum of transmitted packets $\sum_{j\in n_i} r_{ij}^k(t)$ is added to the local queue -- or subtracted if this quantity is negative. Therefore, the number $q_i^k(t)$ of $k$-flow packets queued at node $i$ evolves according to 
\begin{equation}\label{eqn_queue_evolution}
   q_i^k(t+1) = \bigg[ q_i^k(t) + a_i^k(t) + \sum_{j\in n_i} 
       r_{ji}^k(t) - r_{ij}^k(t)\bigg]^+ ,
\end{equation}
where the projection $[\cdot]^+$ into the nonnegative reals is necessary because the number of packets in queue cannot become negative. We remark that \eqref{eqn_queue_evolution} is stated for all nodes $i\neq o_k$ because packets routed to their destinations are removed from the system.

To ensure packet delivery it is sufficient to guarantee that all queues $q_i^k(t+1)$ remain stable. In turn, this can be guaranteed if the average rate at which packets exit queues do not exceed the rate at which packets are loaded into them. To state this formally observe that the time average limit of arrivals staisfies $\lim_{t\to\infty}a_i^k(t)=\E{a_i^k(t)} := a_i^k$ and define the ergodic limit $r_{ij}^k:=\lim_{t\to\infty}r_{ij}^k(t)$. If the processes $r_{ij}^k(t)$ controlling the movement of information through the network are asymptotically stationary, queue stability follows if \footnote{Stability is guaranteed only if the inequalities hold in a strict sense, i.e., $\sum_{j\in n_i} r_{ij}^k - r_{ji}^k > a_i^k $. Equality is allowed here to facilitate connections with optimization problems to be considered later on.}
\begin{equation}\label{dconstraint}
   \sum_{j\in n_i} r_{ij}^k - r_{ji}^k \geq a_i^k \quad  \forall\ k, i \neq o_k.
\end{equation}
For future reference define the vector $r:=\{r_{ij}^k\}_{k,i\not= o_k,j}$ grouping variables $r_{ij}^k$ for all flows and links.  Since at most $C_{ij}$ packets can be transmitted in the link $(i,j)$ the routing variables $r_{ij}^k(t)$ must satisfy
\begin{equation}\label{eqn_capacity_constraint}
   \sum_{k} r_{ij}^k (t) \leq C_{ij}.
\end{equation}
The joint routing and scheduling problem can be now formally stated as the determination of nonnegative variables $r_{ij}^k (t)\geq0$ that satisfy \eqref{eqn_capacity_constraint} for all times $t$ and whose time average limits $r_{ij}^k$ satisfy \eqref{dconstraint}.
The BP algorithm solves this problem by assigning all the capacity of the link $(i,j)$ to the flow with the largest queue differential $q_i^k(t)-q_j^k(t)$. Specifically, for each link we determine the flow pressure
\begin{equation}\label{eqn_queue_pressure}
   k_{ij}^*=\argmax_k \left[q_i^k(t)-q_j^k(t)\right]^+ .
\end{equation}
If the maximum pressure $\max_k\left[q_i^k(t)-q_j^k(t)\right]^+>0$ is strictly positive we set $r_{ij}^k (t)=C_{ij}$ for $k=k_{ij}^*$. Otherwise the link remains idle during the time frame. The resulting algorithm is summarized in Algorithm \ref{BPalg}. The backpressure algorithm works by observing the queue differentials on each link and then assigning the capacity for each link to the data type with the largest positive queue differential, thus driving the time average of the queue differentials to zero-- stabilizing the queues.  For the generalizations developed in this paper it is necessary to reinterpret BP as a dual stochastic subgradient descent as we do in the following section.
%

\subsection{Dual stochastic subgradient descent} \label{ssgd}
Since the parameters that are important for queue stability are the time averages $r_{ij}^k$ of the routing variables $r_{ij}^k(t)$ an alternative view of the joint routing and scheduling problem is the determination of variables $r_{ij}^k$ satisfying \eqref{dconstraint} and $\sum_{k} r_{ij}^k \leq C_{ij}$. This can be formulated  as the solution of an optimization problem. Let $f_{ij}^k(r_{ij}^k)$ be arbitrary concave functions on $\field{R}_+$ and consider the optimization problem 
\begin{alignat}{4} \label{bpopt}
    r^* := 
    &\argmax\ && \sum_{k,i\neq o_i, j} f_{ij}^k(r_{ij}^k)  \\
    &\st      && \sum_{j\in n_i} r_{ij}^k - r_{ji}^k \ge a_i^k, \quad &&\forall\ k, i \neq o_k, \nonumber\\
              &&&\sum_{k\in\ccalK} r_{ij}^k \le C_{ij},  \quad &&\forall\ (i,j)\in\mathcal{E}.  \nonumber
\end{alignat}
where the optimization is over nonnegative variables $r_{ij}^k\geq0$. Since only feasibility is important for queue stability, solutions to \eqref{bpopt} ensure stable queues irrespectively of the objective functions $f_{ij}^k(r_{ij}^k)$.

Since the problem in \eqref{bpopt} is concave it can be solved by descending on the dual domain. Start by associating multipliers $\lambda_i^k$ with the constraint $\sum_{j\in n_i} r_{ij}^k - r_{ji}^k \ge a_i^k$ and keep the constraint $\sum_{k} r_{ij}^k \leq C_{ij}$ implicit. The corresponding Lagrangian associated with the optimization problem in \eqref{bpopt} is
\begin{equation} \label{lagrange}
   \ccalL (r, \lambda) = \sum_{k,  i\not= o_k, j} f_{ij}^k(r_{ij}^k) + \sum_{k,  i\not= o_k} \lambda_i^k
       \bigg(\sum_{j\in n_i} r_{ij}^k - r_{ji}^k - a_i^k\bigg)
\end{equation}
where we introduced the vector $\lambda:=\{\lambda_{i}^k\}_{k, i\not= o_k}$ grouping variables $\lambda_{i}^k$ for all flows and nodes. The corresponding dual function is defined as $h(\lambda):=\max_{\sum_{k} r_{ij}^k \leq C_{ij}} \ccalL (r, \lambda)$. 

To compute a descent direction for $h(\lambda)$ define the primal Lagrangian maximizers for given $\lambda$ as $r_{ij}^k(\lambda):=\argmax_{\sum_{k} r_{ij}^k \leq C_{ij}} \ccalL (r, \lambda)$. A descent direction for the dual function is available in the form of the dual subgradient whose components $\tdg_i^k(\lambda)$ are obtained by evaluating the constraint slack associated with the Lagrangian maximizers
\begin{align}\label{eqn_subgradient}
    \tdg_i^k(\lambda):=\sum_{j\in n_i} r_{ij}^k(\lambda) - r_{ji}^k(\lambda) - a_i^k.
\end{align}
Since the Lagrangian $\ccalL (r, \lambda)$ in \eqref{lagrange} is linear in the routing variables $r_{ij}^k$ the determination of the maximizers $r_{ij}^k(\lambda):=\argmax_{\sum_{k} r_{ij}^k \leq C_{ij}} \ccalL (r, \lambda)$ can be decomposed into the maximization of separate summands. Considering the coupling constraints $\sum_{k} r_{ij}^k \leq C_{ij}$ it suffices to consider variables $\{r_{ij}^k\}_k$ for all flows across a given link. After reordering terms it follows that we can compute $r_{ij}^k(\lambda)$ as
\begin{alignat}{3}
   r_{ij}^k(\lambda)  \label{eqn_primal_maximizers}
      = &\argmax\ && \sum_k f_{ij}^k(r_{ij}^k)+ r_{ij}^k\big(\lambda_i^k-\lambda_j^k\big)\\ 
        & \st     && \sum_{k\in \ccalK } r_{ij}^k \le C_{ij}. \nonumber
\end{alignat}
Introducing a time index $t$, subgradients $\tdg_i^k(\lambda(t))$ could be computed using \eqref{eqn_subgradient} with Lagrangian maximizers $r_{ij}^k(\lambda(t))$ given by \eqref{eqn_primal_maximizers}. A subgradient descent iteration could then be defined to find the variables $r^*$ that solve \eqref{bpopt}; see e.g., \cite{wireless}. 

The problem in computing $\tdg_i^k(\lambda)$ is that we don't know the average arrival rates $a_i^k$. We do observe, however, the instantaneous rates $a_i^k(t)$ that are known to satisfy $\E{a_i^k(t)} = a_i^k$. Therefore,
\begin{align}\label{eqn_stoch_subgradient}
    g_i^k(\lambda):=\sum_{j\in n_i} r_{ij}^k(\lambda) - r_{ji}^k(\lambda) - a_i^k(t),
\end{align}
is a stochastic subgradient of the dual function in the sense that its expected value $\E{g_i^k(\lambda)} = \tdg_i^k(\lambda)$ is the subgradient defined in \eqref{eqn_subgradient}. We can then minimize the dual function using a stochastic subgradient descent algorithm. At time $t$ we have multipliers $\lambda(t)$ and determine Lagrangian maximizers $r_{ij}^k(t):= r_{ij}^k(\lambda(t))$ as per \eqref{eqn_primal_maximizers}. We then proceed to update multipliers along the stochastic subgradient direction according to
\begin{equation}\label{eqn_stoch_subg_descent}
   \lambda_i^k(t+1) = \bigg[ \lambda_i^k(t) - 
      \eps\bigg(\sum_{j\in n_i} r_{ij}^k(t) - r_{ji}^k(t) - a_i^k(t)\bigg)\bigg]^+, 
\end{equation}
where $\eps$ is a constant stepsize chosen small enough so as to ensure convergence; see e.g., \cite{sbp}.

Properties of the descent algorithm in \eqref{eqn_stoch_subg_descent} vary with the selection of the functions $f_{ij}^k(r_{ij}^k)$. Two cases of interest are when $f_{ij}^k(r_{ij}^k)=0$ and when $f_{ij}^k(r_{ij}^k)$ are continuously differentiable, strongly convex, and monotone decreasing on $\field{R}_+$ but otherwise arbitrary. The former allows an analogy with the backpressure as given by Algorithm \ref{BPalg} while the latter leads to a variation termed soft backpressure. 
%
\begin{algorithm}[t]\label{BPalg_as_subgradient_descent}{\small
\caption{Dual stochastic subgradient descent}
\For{$t = 0, 1, 2, \cdots$} {
\For{all neighbors $j\in n(i)$}{ 
Send multipliers $\{\l^k_i(t)\}_{k}$ -- Receive multipliers $\{\l^k_j(t)\}_{k}$\\
Determine flow with largest dual variable differential:
$$k_{ij}^*=\argmax_k \left[\lambda_i^k(t)-\lambda_j^k(t)\right]^+ $$\\
Set routing variables to $r_{ij}^{k}(t) = 0$ for all $k\neq k_{ij}^*$ and
  $$r_{ij}^{k_{ij}^*} (t)= C_{ij} \ind{\lambda_i^{k_{ij}^*}(t)-\lambda_j^{k_{ij}^*}(t)>0} $$
Transmit $r_{ij}^{k_{ij}^*}(t)$ packets for flow $k_{ij}^*$}
Send variables $\{r^k_{ij}(t)\}_{kj}$ -- Receive variables $\{r^k_{ji}(t)\}_{kj}$\\
Update multipliers $\{\l^k_i(t)\}_{k}$ along stochastic subgradient
\[\l^k_i(t+1)=\bigg[ \lambda_i^k(t) -\eps\bigg(\sum_{j\in n(i)} r_{ij}^k(t) + r_{ji}^k(t) + a_i^k(t)\bigg)\bigg]^+\] }}
\end{algorithm}
\subsection{Backpressure as stochastic subgradient descent} 
The classical backpressure algorithm, \cite{backpress} can be recovered by setting $f_{ij}^k(r_{ij}^k)=0$ for all links flows $k$ and links $(i,j)$. With this selection the objective to be maximized in \eqref{eqn_primal_maximizers} becomes $\sum_kr_{ij}^k\left(\lambda_i^k(t)-\lambda_j^k(t)\right)$. To solve this maximization it suffices to find the flow with the largest dual variable differential 
\begin{equation}\label{eqn_multiplier_pressure}
   k_{ij}^*=\argmax_k \left[\lambda_i^k(t)-\lambda_j^k(t)\right]^+ .
\end{equation}
If the value of the corresponding maximum is nonpositive, i.e., $\max_k \left[\lambda_i^k(t)-\lambda_j^k(t)\right]^+\leq0$, all summands in \eqref{eqn_primal_maximizers} are nonpositive and the largest objective in \eqref{eqn_primal_maximizers} is attained by making $r_{ij}^k(t)=0$ for all flows $k$. Otherwise, since the sum of routing variables $r_{ij}^k(t)$ must satisfy $\sum_{k\in\ccalK}r_{ij}^k\le C_{ij}$ the maximum objective is attained by making $r_{ij}^k(t)= C_{ij}$ for $k = k_{ij}^*$ and $r_{ij}^k(t)=0$ for all other $k$.

{The algorithm that follows from the solution of this maximization is summarized in Algorithm \ref{BPalg_as_subgradient_descent}.}  The dual stochastic subgradient descent is implemented using node level protocols.  At each time instance nodes send their multipliers $\l_i^k(t)$ to their neighbors.  After receiving multiplier information from its neighbors, each node can compute the multiplier differentials $\l_i^k(t) -\l_j^k(t)$ for each edge.  The node then allocates the full capacity of each outgoing link to which ever commodity has the largest differential.  Once the transmission rates are set each node can observe its net packet gain which is equivalent to the stochastic gradient as defined in \eqref{eqn_stoch_subgradient}.  Finally, each node updates its multipliers but subtracting $\epsilon$ times ties stochastic subgradient from its current multipliers. Choosing the stepsize $\epsilon=1$ causes the multipliers to coincide with the queue lengths for all time.  

Comparing \eqref{eqn_queue_pressure} with \eqref{eqn_multiplier_pressure} we see that the assignments of flows to links are the same if we identify multipliers $\lambda_i^k(t)$ with queue lengths. Furthermore, if we consider the Lagrangian update in \eqref{eqn_stoch_subg_descent} with $\eps=1$ we see that they too coincide. Thus, we can think of backpressure as a particular case of stochastic subgradient descent. From that point of view algorithms \ref{BPalg} and \ref{BPalg_as_subgradient_descent} are identical. The only difference is that since queues take the place of multipliers the update in Step 8 of Algorithm \ref{BPalg_as_subgradient_descent} is not necessary in Algorithm \ref{BPalg}.
%
%
\begin{algorithm}[t]{\small \caption{Soft Backpressure at node $i$}\label{SBPalg}
\For{$t = 0, 1, 2, \cdots$} {
\For{all neighbors $j\in n(i)$}{ 
Send multipliers $\{\l^k_i(t)\}_{k}$ -- Receive multipliers $\{\l^k_j(t)\}_{k}$\\
Compute $\mu_{ij}$ such that \[\sum_k [\lambda_i^k(t)-\lambda_j^k(t)-\mu_{ij}]^+\] \\
Transmit packets at rate \[r_{ij}^{k}(t)= F(- [\lambda_i^k(t)-\lambda_j^k(t)-\mu_{ij}]^+) \] \\}
Send variables $\{r^k_{ij}(t)\}_{kj}$ -- Receive variables $\{r^k_{ji}(t)\}_{kj}$\\
Update multipliers $\{\l^k_i(t)\}_{k}$ along stochastic subgradient
\[\l^k_i(t+1)=\bigg[ \lambda_i^k(t) -\eps\bigg(\sum_{j\in n_i} r_{ij}^k(t) - r_{ji}^k(t) - a_i^k(t)\bigg)\bigg]^+\] }}
\end{algorithm}
%
%
\subsection{Soft backpressure} \label{ssbp} Assume now that the functions $f_{ij}^k(r_{ij}^k)$ are continuously differentiable, strongly convex, and monotone decreasing on $\field{R}_+$ but otherwise arbitrary. In this case the derivatives $\partial f_{ij}^k(x)/\partial x$ of the functions $f_{ij}^k(x)$ are monotonically increasing and thus have inverse functions that we denote as 
\begin{equation}\label{F}
   F_{ij}^k(x) := \left[\partial f_{ij}^k(x)/\partial x\right]^{-1} (x) .
\end{equation}
The Lagrangian maximizers in \eqref{eqn_primal_maximizers} can be explicitly written in terms of the derivative inverses $F_{ij}^k(x)$. Furthermore, the maximizers are unique for all $\lambda$ implying that the dual function is differentiable. We detail these two statements in the following proposition.
\begin{proposition}\label{prop_explicit_primal_maximizers}
If the functions $f_{ij}^k(r_{ij}^k)$ in \eqref{bpopt} are continuously differentiable, strongly convex, and monotone decreasing on $\field{R}_+$, the dual function $h(\lambda):=\max_{\sum_{k} r_{ij}^k \leq C_{ij}} \ccalL (r, \lambda)$ is differentiable for all $\lambda$. Furthermore, the gradient component along the $\lambda_i^k$ direction is $g_i^k(\lambda)$ as defined in \eqref{eqn_stoch_subgradient} with 
\begin{equation}\label{primal} 
   r_{ij}^k(\lambda) = F_{ij}^k\left(-\left[\lambda_i^k-\lambda_j^k - {\mu_{ij}(\lambda)}\right]^+\right). 
\end{equation}
where $\mu_{ij}(\lambda)$ is either $0$ if $\sum_k F_{ij}^k\big(-\big[\lambda_i^k-\lambda_j^k \big]^+\big)\leq C_{ij}$ or chosen as the solution to the equation
\begin{equation}\label{mu1}
   \sum_k F_{ij}^k\left(-\left[\lambda_i^k-\lambda_j^k - 
        \mu_{ij}(\lambda)\right]^+\right) = C_{ij}.
\end{equation}
\end{proposition}
\begin{proof} The dual problem $\min_\l h(\l)$ is convex because the primal \eqref{bpopt} is convex.  Convex optimization problems have unique maximizers. In our case, the dual gradient comes from differentiating \eqref{lagrange} with respect to $\l_i^k$ which yields \eqref{eqn_stoch_subgradient}. Substituting the unique maximizers we have the unique dual gradient.  Since $h(\lambda)$ has a unique gradient, it is differentiable.

In order to find the previously mentioned unique maximizers we consider the primal optimization \eqref{eqn_primal_maximizers}.  We dualize the additional constrain to get the extended Lagrangian
\begin{equation}
\bar{\mathcal{L}}(r, \mu) = \sum_k f(r_{ij}^k)+ r_{ij}^k(\l_i^k-\l_j^k) +\mu_{ij} (C_{ij}-\sum_k r_{ij}^k) \label{softl}.
\end{equation}
Considering the KarushÐKuhnÐTucker (KKT) optimality conditions as defined in \cite{boydbook}[Section 5.5] for \eqref{softl} yields the equations
\begin{eqnarray}
f'(r_{ij}^k) &=& -\left[\l_i^k-\l_j^k - \mu_{ij}\right]^+ \label{prim1}\\
\sum_k r_{ij}^k&\le& C_{ij} \label{prim2}
\end{eqnarray}
for all $(i,j)\in \mathcal{E}$. Applying the definition of $F_{ij}^k(\cdot)$ from equation \eqref{F} to \eqref{prim1} we get the desired relation in \eqref{primal}.  It remains to enforce \eqref{prim2} by selection of $\mu_{ij}$ which gives us condition \eqref{mu1}.  The assertion that $\mu_{ij}=0$ when $\sum_k F_{ij}^k\big(-\big[\lambda_i^k-\lambda_j^k \big]^+\big)\leq C_{ij}$ hold by the principal of complementary slackness, detailed in \cite{boydbook}[Section 5.5].
\end{proof}
%
While \eqref{mu1} does not have a closed for solution it can be computed quickly numerically using a binary search because it is a simple single variable root finding problem.  Computation time cost remains small compared to communication time cost.  

This yields the Algorithm \ref{SBPalg}.  Sott backpressureis implemented using node level protocols.  At each time instance nodes send their multipliers $\l_i^k(t)$ to their neighbors.  After receiving multiplier information from its neighbors, each node can compute the multiplier differentials $\l_i^k(t) -\l_j^k(t)$ for each edge.  The nodes then solve for $\mu_{ij}$ on each of its outgoing edges by using a rootfinder to solve the local constraint in \eqref{mu1},  The capacity of each edge is then allocated to the commodities via reverse waterfilling as defined in \eqref{primal}. Once the transmission rates are set each node can observe its net packet gain which is equivalent to the stochastic gradient as defined in \eqref{eqn_stoch_subgradient}.  Finally, each node updates its multipliers but subtracting $\epsilon$ times ties stochastic subgradient from its current multipliers. Choosing the stepsize $\epsilon=1$ causes the multipliers to coincide with the queue lengths for all time.  

Algorithm \ref{SBPalg}, is called soft backpressure because rather than allocate all of an edges capacity to the data type with the largest pressure $\lambda_i^k-\lambda_j^k$, it divides the capacity among the different data types based on their respective pressures via inverse water filling, \cite{sbp}.   We further observe that soft backpressure is solved using the same information required of backpressure, the lengths of the queues for all data types on either end of the link whose flow you are computing.  This means that assuming a step size $\epsilon=1$, Algorithm \ref{SBPalg} can be implemented without computing the multipliers but rather by simply observing the queue lengths $q_i^k(t)$ which are equal to the dual variables $\lambda_i^k(t)$. 

An important difference between backpressure and soft backpressure is that the former is equivalent to stochastic {\it subgradient} descent whereas the latter is tantamount to stochastic {\it gradient} descent because the dual function is differentiable. This improves the average convergence rate from logarithmic -- expected distance to optimal variables decreasing like $c/t$ for some constant $c$ -- to linear -- expected distance to optimality proportional to $c^t$ for some $c$. Linear convergence is still not satisfactory, however, motivating the accelerated backpressure algorithm that we introduce in the following section.
%
%
%
\section{Accelerated Backpressure}\label{abps}
Backpressure type algorithms exhibit slow convergence because they are fundamentally first order methods.  In the centralized deterministic case we can speed up using the Newton method.  In our case, the stochastic nature of the problem is not an issue because the arrival rates $a_i^k$ do not appear in the Hessian 
\begin{equation}
H =  \left[\begin{array} {cccc} \label{hessblock}
H_{11} & H_{12} & \cdots & H_{1n}\\
H_{21} & \ddots & & \vdots  \\
\vdots & & \ddots &  \vdots\\
H_{n1} & \cdots & \cdots & H_{nn}
\end{array}\right] \end{equation}
where
\begin{equation}\label{hessele}
[H_{ij}]_{ks} =  \frac{\partial^2 \mathcal{L}(r(\lambda), \lambda) }{\partial \lambda_i^k \partial \lambda_j^s}.
\end{equation}
Assuming knowledge of $H$ and the gradient $g= \{g_i^k\}$ for all $i$ and $k$, the Newton algorithm is given by Algorithm \ref{newt}.
\begin{algorithm}[t]{\small
\caption{Newton's Method}\label{newt}
\For{$t = 0, 1, 2, \cdots$} {
\For{all nodes $j\in{n(i)}$}{
Send multipliers $\{\l_i^k(t)\}_k$ - Receive multipliers $\{\l_j^k(t)\}_k$\\
Compute $\mu_{ij}$ such that \[\sum_k [\lambda_i^k(t)-\lambda_j^k(t)-\mu_{ij}]^+\] \\
Transmit packets at rate \[r_{ij}^{k}(t)= F(- [\lambda_i^k(t)-\lambda_j^k(t)-\mu_{ij}]^+) \] \\
Send variables $\{r^k_{ij}(t)\}_{kj}$-- Receive variables $\{r^k_{ji}(t)\}_{kj}$ \\
Compute $H_{ij}=[\nabla^2 \mathcal{L}(\lambda, r(\lambda))]_{ij}$\\ 
}
Comptute stochastic subgradient $g_i=\{g_i^k\}_k$
\[g_i^k=\sum_{j\in n(i)} r_{ij}^k(t) - r_{ji}^k(t) - a_i^k(t)\]
 \\
\For{all nodes $j\in\mathcal{V}$}{
Send variables $H_{ij}$, $g_i$-- Receive variables $H_{ji}$, $g_j$\\
Compute Hessian inverse blocks $[H^{-1}]_{ij}$\\
}
Update multipliers $\l_i(t)=\{ \l_i^k(t) \}_k$ along the Newton direction
\[ \l_i(t+1)=\left[ \lambda_i(t) -\sum_j [H^{-1}]_{ij}g_j\right]^+ \] 
}
}
\end{algorithm}
%
%
\begin{proposition}
The Dual Hessian, $H(\lambda) =  \nabla^2 \mathcal{L}(r(\lambda), \lambda)$ is block sparse with respect to the graph $\mathcal{G}$:
\begin{equation}
H_{ij} = \mathbf{0} \qquad \forall i\not =j \hbox{ s.t } (i,j), (j,i)\not \in \mathcal{E} 
\end{equation}
and the diagonal blocks are $H_{ii}$ are positive semidefinite.
\label{prop}
\end{proposition}
\begin{proof}
%
%
In order to compute the dual Hessian, we begin by computing the optimal flow rates  $r_{ij}^k(\l)$ from the optimal queue priorities as defined in \eqref{primal}.  Substituting into
\begin{equation}\label{eqn_dual_hessian_definition}
\frac{\partial \mathcal{L}(r(\lambda), \lambda) }{\partial \lambda_i^k}=\sum_{j\in n_i} 
       r_{ji}^k(t) - r_{ij}^k(t) - a_i^k
\end{equation}
and differentiating with respect to $\lambda$ we construct the Hessian.  Since $a_i^k$ is a constant the key is differentiating $r_{ij}^k(\lambda)$ with respect to $\l$.  We can differentiate \eqref{primal} using the chain rule yielding
\begin{equation}\label{chain}
\frac{\partial  r_{ij}^k(\lambda)}{\partial \lambda_i^k} = \frac{\partial F_{ij}^k(x)}{\partial x} \frac{\partial}{\partial \l_i^k} \left(-\left[\lambda_i^k-\lambda_j^k - {\mu_{ij}}\right]^+\right)
\end{equation}
The existence of  ${\partial F_{ij}^k(x)}/{\partial x}$ is guaranteed by our assumptions on the edge costs $f_{ij}^k$.  Differentiating $-\left[\lambda_i^k-\lambda_j^k - {\mu_{ij}}\right]^+$ is done by observing that this function becomes a saturated linear function of $r_{ij}^k$ when $r_{ij}^k$ is non-zero and exactly zero when $r_{ij}^k=0$. The slope of the linear function is determined by whether the edge $(i,j)$ is at capacity ($\mu_{ij}>0$).  Discontinuities in $r_{ij}^k$ occur when $r_{ij}^k=0$ and when the projection becomes active, resulting in points from the derivative. We define the derivatives to be zero out these points since they are adjacent to regions where the derivatives are defined and equal to zero. We express the derivative in terms of the active sets
\begin{equation}
\mathcal{A}_{ij}=\{k\in\mathcal{K} : r_{ij}^k(\l)>0\} \label{action}
\end{equation}
for each edge $(i,j)$ and the local variables $\mu_{ij}$ and $\{r_{ij}^k\}$.  Using \eqref{chain} the diagonal elements $H_{ii}$ are defined
\begin{eqnarray}
[H_{ii}]_{kk} &=& \sum_{j\in n_i}  F'(r_{ij}^k)\1(k\in \A_{ij})\left(\frac{\1(\mu_{ij}>0)}{|\A_{ij}|}-1\right)\nonumber\\&+& F'(r_{ji}^k)\1(k\in \A_{ji})\left(\frac{\1(\mu_{ji}>0)}{|\A_{ji}|}-1\right).\label{onon}
\end{eqnarray}
Using same method the off diagonal elements of $H_{ii}$ are computed
\begin{eqnarray}
[H_{ii}]_{ks} &=& \sum_{j\in n_i}  F'(r_{ij}^k)\1(k,s\in \A_{ij})\left(\frac{\1(\mu_{ij}>0)}{|\A_{ij}|}\right)\nonumber\\&+& F'(r_{ji}^k)\1(k,s\in \A_{ji})\left(\frac{\1(\mu_{ji}>0)}{|\A_{ji}|}\right).\label{onoff}
\end{eqnarray}
We observe that $\sum_{j \in n_i} [H_{ij}]_{kk} = [H_{ii}]_{kk}$ from the definitions in \eqref{onon} and \eqref{onoff}. Also, $F'(r_{ij}^k)\le 0$ from {\eqref{F}} and the our assumption that $f_{ij}^k(\cdot)$ is monotone decreasing, therefore the diagonal elements are positive and according to {\cite{NonNegativeMatrices}[Section 6.2]} $H_{ii}$ is positive semidefinite.

We can also compute the off diagonal blocks by differentiating using \eqref{chain}. The diagonal elements of the the off-diagonal blocks $H_{ij}$ are given by
\begin{eqnarray}
[H_{ij}]_{kk} &=& F'(r_{ij}^k)\1(k\in \A_{ij})\left(1-\frac{\1(\mu_{ij}>0)}{|\A_{ij}|}\right)\nonumber\\&+&F'(r_{ji}^k) \1(k\in \A_{ji})\left(1-\frac{\1(\mu_{ji}>0)}{|\A_{ji}|}\right)\label{offon}
\end{eqnarray}
and the off-diagonal elements of the off-diagonal blocks are
\begin{eqnarray}
[H_{ij}]_{ks} &=& -F'(r_{ij}^k)\1(k,s\in \A_{ij})\left(\frac{\1(\mu_{ij}>0)}{|\A_{ij}|}\right)\nonumber\\&-&F'(r_{ji}^k) \1(k,s\in \A_{ji})\left(\frac{\1(\mu_{ji}>0)}{|\A_{ji}|}\right).\label{offoff}
\end{eqnarray}
Observe that the action set $\A_{ij}$ is empty by definition if there is no link $(i,j)\in \mathcal{E}$, therefore $k$ cannot be an element of $\A_{ij}$.  Plugging this fact into the definition of the off diagonal blocks $H_{ij}$ found in equations \eqref{offon} and \eqref{offoff} we conclude that $H_{ij}=\mathbf{0}$ whenever $i\not =j$ and $(i,j), (j,i)\not \in \mathcal{E}.$
\end{proof}

Proposition 2 guarantees that all elements of the Hessian can be computed using local information.   Elements of the Hessian require knowledge of the local action sets $\A_{ij}$ from \eqref{action}, which are computed using the local flow values $\{r_{ij}^k\}_k$.
Furthermore, Proposition 2 gives us positive semi-definiteness of the diagonal blocks $H_{ii}$ indicating that the nodes depend positively on their own queues. This  fits our intuition because we expect penalties on a specific queue to become larger when those queues become larger. Later we will use this fact to show that our splitting matrix $\bar D$ is invertable.

Algorithm \ref{newt} is centralized because it requires the inverse of the Hessian matrix. For comparison with our algorithms we describe node level protocols for implementing Newton's method for the backpressure problem.  At each time instance nodes send their multipliers $\l_i^k(t)$ to their neighbors.  After receiving multiplier information from its neighbors, each node can compute the multiplier differentials $\l_i^k(t) -\l_j^k(t)$ for each edge.  The nodes then solve for $\mu_{ij}$ on each of its outgoing edges by using a rootfinder to solve the local constraint in \eqref{mu1},  The capacity of each edge is then allocated to the commodities via reverse waterfilling as defined in \eqref{primal}. Once the transmission rates are set each node can observe its net packet gain which is equivalent to the stochastic gradient as defined in \eqref{eqn_stoch_subgradient}. The nodes must also compute the matrix of second derivatives, $H_{ij}$ for each of its outgoing edges  according to \eqref{onon}--\eqref{offoff}  which are functions of local transmission rates $\{r_{ij}^k\}_{k, j\in n(i)}$.  In order to compute the Newton direction nodes must then send their sub gradients and $g_i$ and all of their Hessian blocks $H_{ij}$ to \textit{every node} in the Network,  With this information each node can invert the global Hessian matrix and compute its Newton direction $d_i = -\sum_j [H^{-1}]_{ij}g_j$.  Finally, each node updates its multipliers by adding $d_i^k$ to its current multipliers.  In this algorithm the multipliers are no longer equivalent to the queue lengths but can be thought of as "queue priorities" which take into account the queue lengths and the networks ability to handle this queues.
%
%
\subsection{Distributed approximations of Newton steps}
In order to accelerate backpressure and retain its distributed nature we  compute the dual update direction based on the ADD-N algorithm defined in \cite{acc11} which leverages the sparsity of the Hessian to approximate its inverse. The ADD-N algorithm requires a splitting of the Hessian.  We define a block splitting {$H=(D+I)-(B+ I)$ where $D$ is block diagonal defined by
\begin{equation}
D_{ii} = H_{ii} 
\end{equation}
and $B$ is block sparse defined by
\begin{equation}
B= D-H.
\end{equation}}
%
%
The ADD-N method gives us the approximate Hessian inverse
\begin{equation}
H^{-1} \approx \bar H^{(N)} = \sum_{\tau=0}^N \left( (D+  I)^{-1} (B+ I)\right)^\tau (D+ I)^{-1}.
\end{equation}
The diagonal blocks $[D+I]_{ii}$ are positive definite for all $i$ because $D_{ii}$ is positive semi-definite from Proposition \ref{prop}.  We define simplified notation for our splitting
\begin{equation}
H = \bar D - \bar B \label{barsplit}
\end{equation}
where $\bar{D} = D +  I$ and $\bar{B} = B+ I$ giving us the expression
\begin{equation}
\bar H^{(N)} =   \sum_{\tau=0}^N \left(\bar D ^{-1} \bar B\right)^\tau \bar D^{-1}.\label{Hbar}
\end{equation}
By construction, the matrix $\bar H^{(N)}$ is block sparse such that $\bar H^{(N)}_{ij}>\mathbf{0}$ only if node $i$ and $j$ are $N$ or fewer hops apart in $\mathcal{G}$.  We are interested in the fully distributed case where only one hop neighbor information is required. Selecting $N=1$, the dual update direction is computed as
\begin{equation}d_i(t) = -\bar D_{ii}^{-1}g_i(t)  -\sum_{j\in n_i} \bar D_{ii}^{-1} \bar B_{ij}\bar D_{jj}^{-1}g_j(t).\label{lan}\end{equation}
The update direction $d_i^k$ can be computed using only local information acquired by a single exchange of information with neighboring nodes as described in Algorithm \ref{add}.  Each node uses knowledge of the flows entering $\{r_{ji}^k\}_k$ and leaving $\{r_{ij}^k\}_k$ to compute the active sets $\mathcal{A}_{ij}$.  Using the active sets and the flows node $i$ can compute the Hessian submatrix $H_{ij}$ for each of its neighbors $j\in n(i)$. The sub matrix $H_{ii}= -\sum_{j\in n(i)} H_{ij}$. From the Hessian submatrices node $i$ can compute the submatrices of our splitting $\bar B_{ij} = -H_{ij}$, $\bar B_{ii}=I$, $\bar D_{ii}= I + H_{ii}$.  The gradient $\{g_i^k\}_k$ is compute from the flows $\{r_{ij}^k\}_k$ and $\{r_{ji}^k\}_k$.  Node $i$ can invert $\bar D$ locally because it is block diagonal. Node $i$ exchanges its local variables $\bar D^{-1}_{ii}$ and $g_i=\{g_i^k\}_k$ with its neighbors $j\in n(i)$ allowing the computation of the approximate newton direction at node $i$, $d_i$ given by equation \eqref{lan}.
%
\begin{algorithm}{\small
\caption{Local Approximation of Newton's Direction}\label{add}
\For{all neighbors $j\in n(i)$}{
Observe local flows $\{r_{ij}^k\}_k$ and $\{r_{ji}^k\}_k$\\ 
Compute active sets\[ \mathcal{A}_{ij}=\{k\in\mathcal{K} : r_{ij}^k(\l)>0\} \label{action}\]\\
Compute $H_{ij}$ according to \eqref{offon} and \eqref{offoff}.\\
Set $\bar B_{ij}=-H_{ij}$.\\
}
Compute $H_{ii}$ according to  \eqref{onon} and \eqref{onoff}.\\
Compute $\bar D_{ii}^{-1} = (H_{ii}+I)^{-1}$\\
Set $\bar B_{ii}=I$\\
Compute $g_i=\{g_i^k\}_k$ according to 
\[g_i^k=   \sum_{j\in n_i} 
       r_{ji}^k(t) - r_{ij}^k(t) - a_i^k(t)\]\\
\For{all neighbors $j\in n(i)$}{
Communicate $\bar D_{ii}^{-1}$ and $g_i $\\
Receive $\bar D_{jj}^{-1}$ and $g_j $\\
}
Compute $d_i= \{d_i^k\}_k$ according to
\[ d_i = -\bar D_{ii}^{-1} g_i - \sum_{j\in n_i} \bar D_{ii}^{-1} \bar B_{ij} \bar D_{jj}^{-1}g_j \]
}
\end{algorithm}
%
%
The dual update for the accelerated backpressure algorithm is given by
\begin{equation}
\lambda_i(t+1) = \lambda_i(t) - \sum_{j\in n_i} \bar H_{ij} g_j \label{dual}
\end{equation}
where $[\lambda_i]_k = \lambda_i^k$ is the local vector of duals at node $i$ and $\bar H_{ij}$ is the $i,j$ block of the matrix $\bar H$.  The dual updates for variables belonging to node $i$ depend only on the variables $r_{ij}^k$ and $g_i^k$ belonging to nodes $j\in n_i$.  
The accelerated backpressure algorithm is given by
\begin{algorithm} [t] \small{\caption{Accelerated Backpressure for node $i$}
\label{ABPalg}
Observe $q_i^k(0)$.\\  
Initialize $\l_i^k(0) = q_i^k(0)$ for all $k$ and $i\not= dest(k)$\\
\For{$t = 0, 1, 2, \cdots$} {
\For{all neighbors $j\in n(i)$}{ 
Send multipliers $\{\l^k_i(t)\}_{k}$ -- Receive multipliers $\{\l^k_j(t)\}_{k}$\\
Compute $\mu_{ij}$ such that \[\sum_k [q_i^k(t)-q_j^k(t)+\beta_{ij}^k-\mu_{ij}]^+\] \\
Transmit packets at rate \[r_{ij}^{k}(t)= F(- [q_i^k(t)-q_j^k(t)+\beta_{ij}^k-\mu_{ij}]^+) \] \\}
Send variables $\{r^k_{ij}(t)\}_{kj}$ -- Receive variables $\{r^k_{ji}(t)\}_{kj}$\\
Compute stochastic gradient $\{g_i^k(t)\}_k$
\[g_i^k(t) =   \sum_{j\in n_i} 
       r_{ji}^k(t) - r_{ij}^k(t) - a_i^k(t)\]\\
Compute $d_i^k$ by executing \textbf{Algorithm \ref{add}}\\
Update the dual variables 
\[\l_i^k(t+1) = \l_i^k(t) +\epsilon d_i^k(t)\]
}}
\end{algorithm}
Algorithm \ref{ABPalg} works like soft backpressure but the set of queue priorities $\l_i^k(t)$ are not equivalent to the queue lengths $q_i^k$.  The queue priorities are updated using information about not only node $i$'s queue but also the queues at neighboring nodes $j\in n_i$.  The addition information is weighted according the approximate inverse Hessian. The cost of this change is that rather than simply observe the queues as they evolve we must observe the realized flows $r_{ij}^k$ and update the queue priorities according.  This requires one additional exchange of information with neighboring nodes, which appears in Algorithm \ref{add}.  However the slightly increased communication overhead of the accelerated backpressure algorithm is offset by significant improvement in convergence rate which we demonstrate in the following sections.
\section{Stability Analysis}
\label{analysis}
In order to claim the accelerated backpressure algorithm is an alternative to the backpressure and soft backpressure algorithms we need to guarantee that it achieves queue stability.  In this section we leverage results from the the stability analysis of the backpressure algorithm detailed extensively in \cite{crosslayer}.  We begin by proving the stability of the dual variables because they are analogous to queue lengths in the backpressure analysis.
\begin{proposition}
\label{queuestability}
Given a Network $\mathcal{G}$ with packet arrivals $\field{E}[a_i^k(t)]=a_i^k$ and edge capacities $C_{ij}$, routing packets according to Algorithm \ref{ABPalg},  we have stable dual variables
\begin{equation}
\limsup_{t\rightarrow\infty} \frac{1}{t} \sum_{\tau=0}^{t-1} \sum_{i,k} \field{E} \l_i^k(\tau)\le \frac{B}{2\delta} \label{dualstab}
\end{equation}
for each commodity $k$ and node $i\not=dest(k)$, where $B,\delta>0$.
\end{proposition}
\begin{proof}
Define the Lyapunov function for the dual vector sequence $\l_t = \{\l_i^k(t)\}_{i,k}$ for all $t>0$,
\begin{equation}
L(\l_t) = \l_t' \bar H_t^{-1} \l_t
\end{equation} 
where $H_t=H(\lambda_t)$.  We consider the variation in $\Delta_t = \bar H_{t+1}^{-1} - \bar H_t^{-1}$.  Due to practical constraints on the flow rates $r_{ij}^k(t)$ defined in \eqref{eqn_primal_maximizers}, and the construction of $\bar H_t$ from \eqref{Hbar} with $N=1$, it is clear that there exists an upper bound of the singled iteration change $\Delta_t$.  We will characterize this bound as $x'\Delta_t x \le b$ for all $t$.  Using the definition of $\Delta_t$ we can write the change in the Lyapunov function
\begin{equation}
L(\l_{t+1})-L(\l_t)  = \l_{t+1}' [\bar H_t^{-1}+\Delta_t] \l_{t+1}-\l_t' \bar H_t^{-1} \l_t
\end{equation}
Reorganizing terms we have 
\begin{equation}
L(\l_{t+1})-L(\l_t)  = \l_{t+1}' \Delta_t  \l_{t+1} + (\l_{t+1}-\l_t)' \bar H_t^{-1} (\l_{t+1} + \l_t)
\end{equation}
Applying equation \eqref{dual} to substitute $\l_{t+1}-\l_t = - \bar H_t g_t$ and simplifying,
\begin{equation}
L(\l_{t+1})-L(\l_t)  = \l_{t+1}' \Delta_t  \l_{t+1} + (-g_t)'  (2\l_t-\bar H_t g_t).
\end{equation}
We next add and subtract $f(r) = \sum_{i,j,k} f_{ij}^k(r_{ij}^k)$ in order to make the primal objective from \eqref{eqn_primal_maximizers} appear,
\begin{eqnarray}
L(\l_{t+1})-L(\l_t)  &= & \l_{t+1}' \Delta_t  \l_{t+1} + g_t \bar H_t g_t +2 f(r)'\nonumber\\
&&- 2(g_t'\l_t+f(r))
\end{eqnarray}
where $g_t$ is the gradient as defined in \eqref{eqn_stoch_subgradient}.  Since the flow values lie in bounded set defined by the capacities $g_t \bar H_t g_t$ and $f(r)$ both have finite maximums.  Along with our finite bound $x'\Delta_t x \le b$ we can conclude $\l_{t+1}' \Delta_t  \l_{t+1} + g_t \bar H_t g_t +2 f(r)\le B$ for some $B>0$.   We apply this inequality and take the expectation with respect to $\mathcal{I}_t$ which is the sigma field including all relevant information up to time $t$ including the dual variables, queue lengths and packet transmission rates:
\begin{equation}
\field{E}[L(\l_{t+1})-L(\l_t) |\mathcal{I}_t] \le B - \field{E}[2(g_t'\l_t+f(r))|\mathcal{I}_t].
\end{equation}
We now applying the argument from \cite{crosslayer}[Theorem 4.5] with \cite{crosslayer}[Corollary 3.9], which (to summarize) uses the fact that $r=\{r_{ij}^k\}$ are the primal optimizers and a small perturbation $\delta>0$ to the arrival rates $\{a_i^k\}$ to lower bound $\field{E}[(g_t'\l_t+f(r))|\mathcal{I}_t] \ge \delta \sum_{i,k} \l_i^k(t)$.  Thus we have 
\begin{equation}
\field{E}[L(\l_{t+1})-L(\l_t) |\mathcal{I}_t] \le B - 2\delta \sum_{i,k} \l_i^k(t) \label{pre},
\end{equation}
the necessary condition for the Lyapunov Stability Lemma, \cite{crosslayer}[Lemma 4.1]  which guarantees the desired relation.
\end{proof}
\begin{corollary} \label{cor1}
Stability of the dual variables $\l_i^k(t)$ implies stability of the queues $q_i^t(t)$.
\end{corollary}
\begin{proof}
The queues evolve according to \eqref{eqn_queue_evolution} while the duals evolve according to \eqref{dual}.  Both updates are based on the realization of the sequence of subgradients in \eqref{eqn_stoch_subgradient}, the difference being the matrix product $\bar H_t g_t$ appearing in \eqref{dual}.  From Proposition 1, we have the Hessian diagonal blocks $H_{ii}$ are positive definite and $H$ is block diagonally dominant. By the construction in \eqref{Hbar}, we can see $\bar H$ also has positive definite diagonal blocks and block diagonal dominant. Therefore, the sequence $\l_i^k(t)$ only remains stable if \eqref{dconstraint} is satisfied, which is the condition for the stability of the queue $q_i^k(t)$.\end{proof}
Proposition \ref{queuestability} guarantees the stability of the dual variables $\l_i^k(t)$ (queue priorities) which take the role the queues have in BP and SBP for determining routing in the ABP algorithm.  Demonstrating that the $\l_i^k(t)$ sequence is stable tells us that the algorithm itself is stable in the sense that the routing assignments stabilize (because they are an explicit function of $\l_i^k(t)$).  It is Corollary \ref{cor1} which guarantees that the queues themselves remain stable. This is done by observing that the evolution of $\l_i^k(t)$ and $q_i^k(t)$ both evolve based on the sequence of dual gradients $g_i^k(t)$.  Since the ABP algorithm solves the optimization \eqref{bpopt}, the dual gradient $g_i^k(t)$ tends to zero on average so we are not surprised to see that both $\l_i^k(t)$ and $q_i^k(t)$ are stable.
%
%
\section{Numerical Results}
\label{sim}
The ABP algorithm can be demonstrated to outperform backpressure and soft backpressure in numerical experiments.  The following examples are run using the 10 node network shown in Figure \ref{net} with 5 data types. Our choice of objective function is 
\begin{equation}
f_{ij}^k(r_{ij}^k) = -\frac{1}{2} \left(r_{ij}^k\right)^2+ \beta_{ij}^k r_{ij}^k.
\end{equation} 
The quadratic term captures an increasing cost of routing larger quantities of packets across a link and help to eliminate myopic routing choices that lead to sending packets in cycles.  The linear term $\beta$ is introduced to reward sending packets to their destinations. In our simulations, $\beta_{ij}^k=10$ for all edges routing to their respective data type destinations $j=dest(k)$ and all other $i,j,k$, $\beta_{ij}^k=0$.  The Link capacities are select uniformly randomly $[0,100]$. 
\begin{figure}
\includegraphics[width=.75\columnwidth]{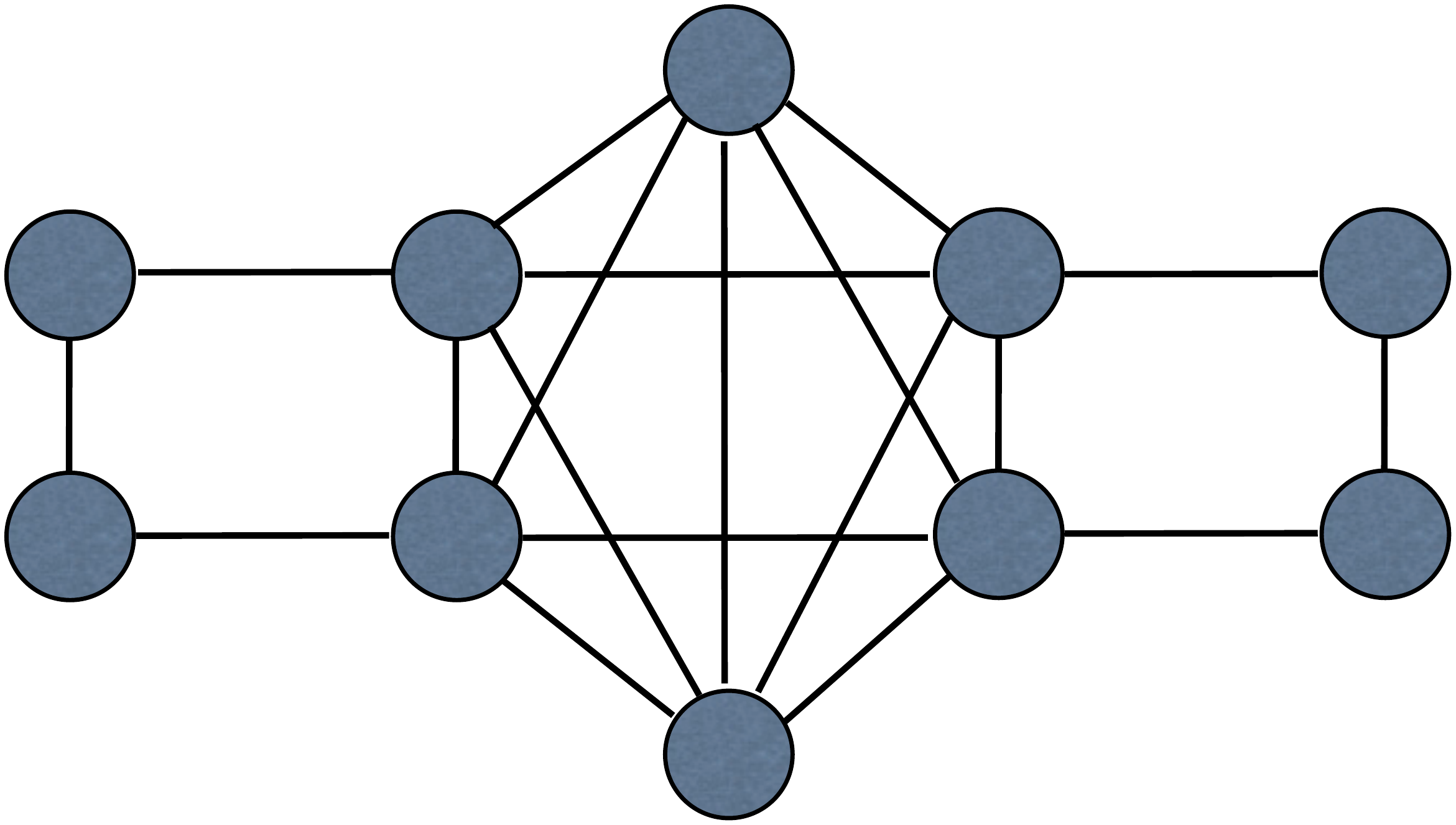}
\centering
\caption{\label{net} The numerical experiments for the ABP algorithm presented in this section are performed on this simple 10 node network with 5 data types. The destinations are unique for each data type and are chosen randomly.}
\end{figure}
%
\begin{figure}
\includegraphics[width=1\columnwidth]{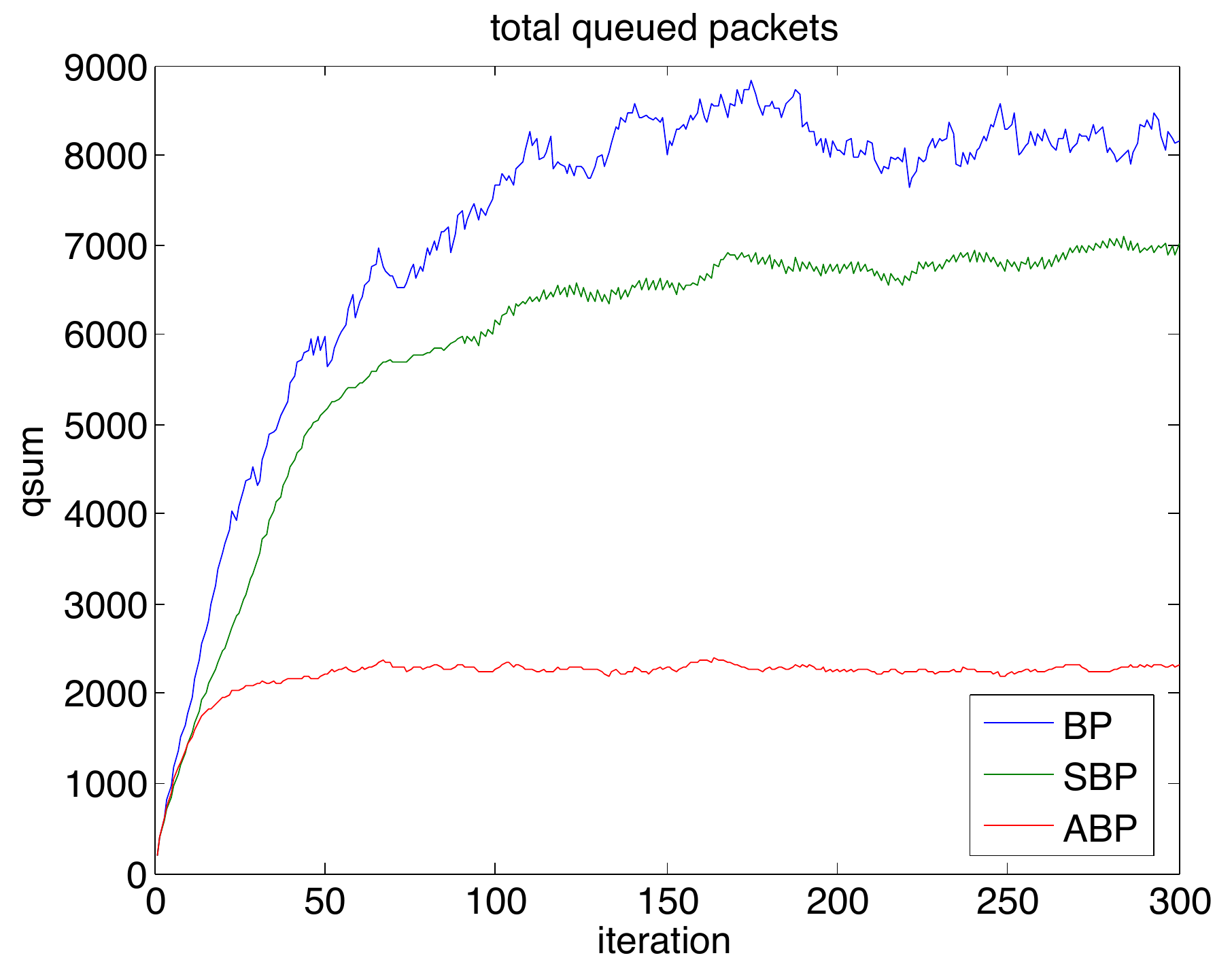}
\centering
\caption{\label{rand} Even when the arrival rates are unknown, the Accelerated Backpressure algorithm stabilizes the queues faster and with smaller queue totals than the Soft Backpressure or the traditional Backpressure algorithms }
\end{figure}
\begin{figure}
\includegraphics[width=1\columnwidth]{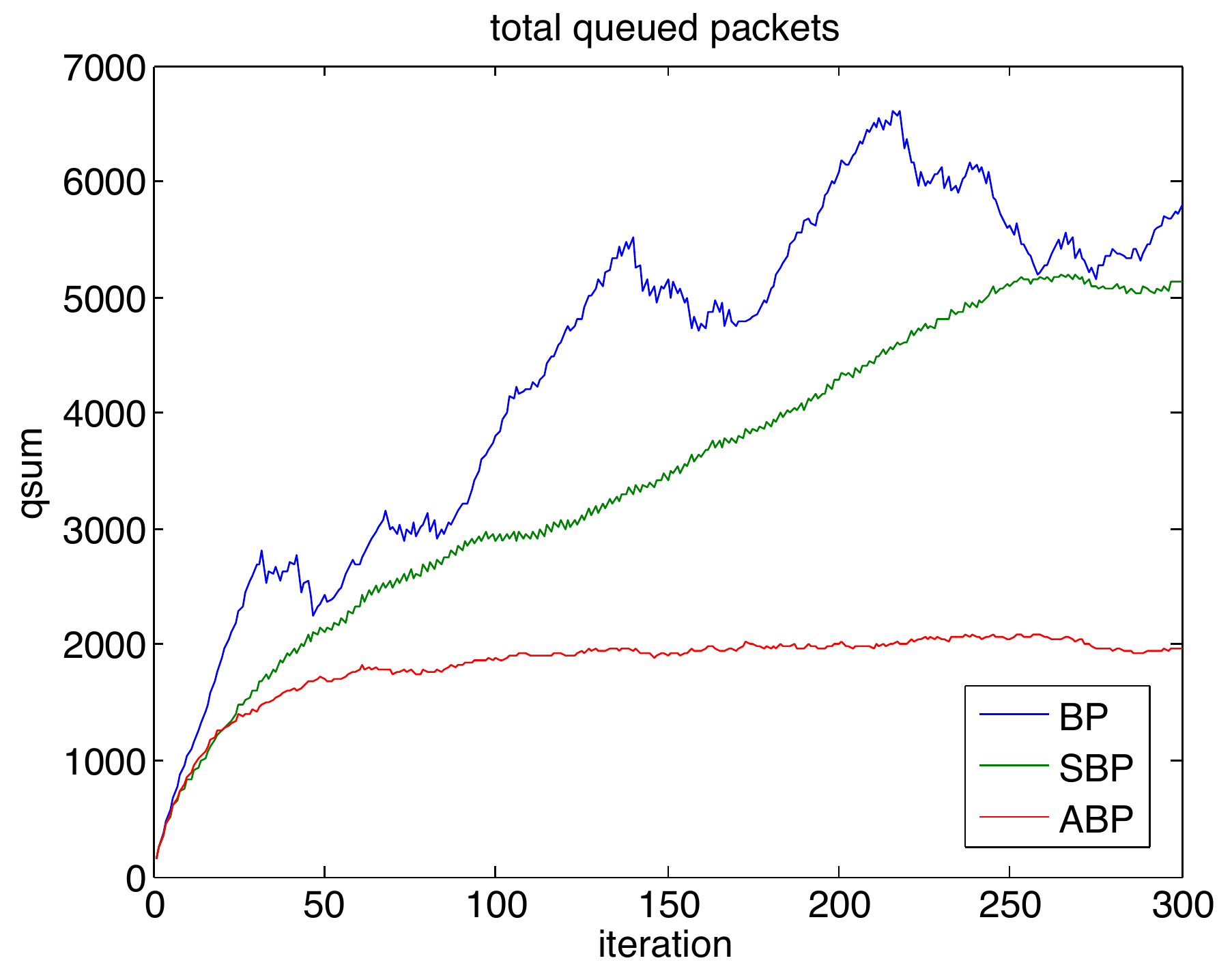}
\centering
\caption{\label{queues}  The variation in the underlying statistics for the packet arrival rates causes the backpressure and soft backpressure algorithm to stabilize much more slowly and in some cases not at all. }
\end{figure}
\begin{figure}
\includegraphics[width=1\columnwidth]{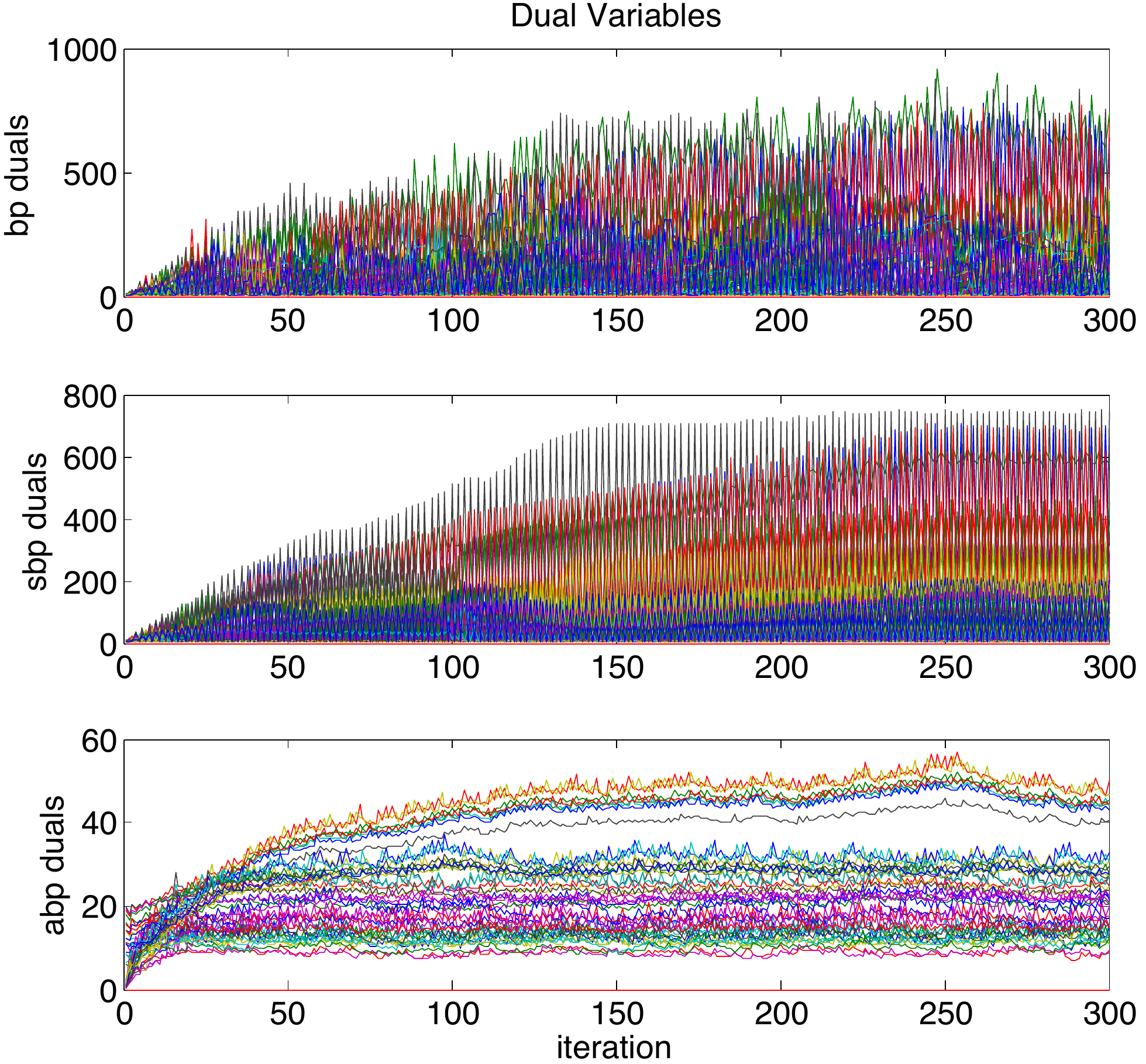}
\centering
\caption{\label{duals}  The Accelerated Backpressure algorithm converges to the optimal dual variables at a rate significantly greater than the backpressure or soft backpressure algorithm, allowing the algorithm to track the optimal duals even when the underlying arrival rates vary. }
\end{figure}
We demonstrate that the ABP algorithm outperforms backpressure and soft backpressure when the arrival rates are stochastic. Numerical experiments are formulated with the the average arrival rates are 5 packets per data type per node.  The nodes do not know that this is the average and therefore use the current realized arrival rate when computing their transmission rates.  As shown in Figure \ref{rand}, the ABP algorithm still stabilizes the queues after about 30 iterations while soft backpressure requires 100 and backpressure requires around 150 iterations but retains the most volatility.  The ABP queues stabilize with around 1000 queued packets while soft backpressure has about 6000 queued packets and the traditional backpressure algorithm has around 8000 queues packets.  In addition to having the smallest queues we observe that the accelerated backpressure algorithm also has the smallest variance in the queues.

In order to demonstrate the effect of having a faster convergence rate we consider the case were the nodes are divided into two equal sized sets. At time $t$ one of the two sets will be actively receiving packets.  The frequency at which we switch between active sets captures how fast the system is varying. The example shown in Figures \ref{duals} and \ref{queues} switches active sets every 10 iterations. Otherwise it is equivalent to the previous examples where 10 nodes were used with the topology shown in Figure \ref{net}.  The key observation in \ref{duals} is that the dual variables are highly volatile for backpressure and soft backpressure.  Recalling that in the case of backpressure and soft backpressure the duals are equivalent to the queue lengths, volatility in the duals is volatility in the queues.  However, accelerated backpressure converges quickly enough that the variations in the underlying statistics at a rate of once every 10 iterations has no significant effect.  From Figure \ref{queues} it is clear that by having much more stable dual variables we are able to compute more effective transmission rate assignments and keep the queues from growing.
\section{conclusion}
In this work we have presented a novel method for computing packrouting in networks based on applying an approximate Newton's method to the backpressure routing problem.  This approach retains the distributed information structure necessary for implementing the algorithm efficiently at the node level. We presented node level protocols and proved that our algorithm stabilizes queues provided the arrival rates and capacities are chosen such that it is possible to stabilize the queues.  In numerical experiments we demonstrate significant improvements in convergence rate leading to much smaller queues and the ability to stabilize queues even when the arrival rate statistic vary.  Our forthcoming work focuses on extending the analysis of the ABP algorithm to explicitly consider the convergence rates.
\bibliographystyle{unsrt}
\bibliography{distributed}
\end{document}